\documentclass[a4paper, 12pt]{article}
\usepackage{latexsym}
\usepackage[T1]{fontenc}
\usepackage{amsmath}
\usepackage{amssymb}
\usepackage{amsthm}
\usepackage{graphicx,graphpap,epsfig,float}
\usepackage{color}
\usepackage{dsfont}
\newtheorem{defn}{Definition}[section]
\newtheorem{lem}[defn]{Lemma}

\newtheorem{theo}[defn]{Theorem}

\theoremstyle{remark}

\theoremstyle{remark}
\newtheorem{ex}[defn]{Example}
\newcommand{\vc}[1]{\boldsymbol{#1}}
\newcommand{\vone}{\vc 1}
\newcommand{\diag}{\mathrm{diag}}
\renewcommand{\sp}{\mathrm{sp}}
\newcommand{\p}{\mathbb{P}}

\renewcommand{\aa}{\mathcal{A}}
\newcommand{\zz}{\mathcal{Z}}

\newcommand{\E}{\mathbb{E}}

\begin{document}

\title{Lyapunov exponents for branching processes in a random environment:  The effect of information}
\author{Sophie Hautphenne\footnote{The University of Melbourne, Department of Mathematics and Statistics, Victoria 3010, Australia; sophiemh@unimelb.edu.au}
\and
Guy Latouche\footnote{%
Universit\'e Libre de Bruxelles, D\'epartement d'Informatique,
CP~212, Boulevard du Triomphe, 1050 Bruxelles, Belgium; latouche@ulb.ac.be}
}

\maketitle

\begin{abstract}
  We consider multitype Markovian branching processes evolving in a
  Markovian random environment.  To determine whether or not the branching
  process becomes extinct almost surely is akin to computing
  the maximal Lyapunov exponent of a sequence of random matrices,
  which is a notoriously difficult problem. We define dual processes and we construct bounds for the  Lyapunov exponent.  The bounds are obtained  by adding or by removing information: to add information results in a lower bound, to remove information results in an upper bound and we show that to add more information gives smaller lower bounds.   We give a few illustrative examples and we observe that the  upper bound is
  generally more accurate than the lower bound.
\end{abstract}

\section{Introduction}

We consider an irreducible multitype Markovian branching process with~$r$ types of individuals and  
we assume that its parameters
vary over time according to a Markovian random environment
$\{X(t):t\in\mathbb{R}^+\}$.  This is an irreducible
continuous-time Markov chain on the finite state space
$E=\{1,2,\ldots,m\}$

In the absence of a random environment, the
parameters of the branching process stay constant over time, and an
individual of type $i$ ($1\leq i \leq r$) lives for an exponentially
distributed amount of time, after which it
generates a random number of children of each type $j$, $1\leq j \leq
r$. It is well known  (Athreya and
Ney~\cite{an72})  that the mean population size matrix $M(t)$, for which $M_{ij}(t)$ is
the conditional expected number of individuals of type $j$ alive at
time $t$, given that the population starts at time 0 with one single
individual of type $i$, is given by $M(t)=e^{\Omega t}$ for some 
matrix $\Omega$ with nonnegative off-diagonal elements.   Furthermore, extinction of the branching process occurs with
probability  one if and only if $\lambda \leq 0$, where 
$\lambda$ is the eigenvalue of maximal real part of $\Omega$.

Here, we associate a matrix $\Omega_\ell$ to each environmental state $\ell$, $\ell=1,\ldots,m$,  and the process $\{X(t)\}$ controls the parameters of the
multitype branching process  in such a way
that the mean population size matrix is $M_\ell(t) := e^{\Omega_\ell t}$ during intervals of time over which
$X(\cdot)=\ell$.

The following theorem gives a necessary and sufficient condition for the almost sure extinction of the branching process in a Markovian random environment.  The notation is as follows:
$\hat\Omega_k= \Omega_{\hat X_k}$ where $\{\hat{X}_k:k\in\mathbb{N}\}$ is the  jump chain associated with the random environment $\{X(t)\}$, 
and $\{\xi_k:k\in\mathbb{N}\}$ is  the sequence of sojourn times in the successive environmental  states.
\begin{theo} {\rm (Tanny~\cite[Theorem 9.10]{tanny81})}  There exists a constant $\omega$ such that 
\begin{equation}\label{mp}\omega = \lim_{n \rightarrow \infty}
  \frac{1}{n} \log \{e^{\hat\Omega_0 \xi_0}\,
  e^{\hat\Omega_1 \xi_1}\,\cdots \,e^{\hat\Omega_{n-1}
    \xi_{n-1}}\}_{ij}
\end{equation} 
almost surely, independently of $i$ and $j$.
Extinction is almost sure if and only if $\omega\leq 0$.
\end{theo}

The $(i,j)$th entry of the random matrix product in \eqref{mp} is the conditional expected number of individuals of type $j$ alive just before the $n$th environmental state change, given that the population starts at time 0 with one single individual of type $i$ and given the history of the environmental process. The limit  $\omega$ may be interpreted as the asymptotic growth rate of the population, and it takes  different forms since one also has
\begin{eqnarray}
   \label{eq:ealpha}
\omega &= &\lim_{n \rightarrow \infty} \frac{1}{n} \E \log
\{e^{\hat\Omega_0 \xi_0}\,
  e^{\hat\Omega_1 \xi_1}\,\cdots \,e^{\hat\Omega_{n-1}
    \xi_{n-1}}\}_{ij} 
\quad \mbox{for all $i$, $j$,}
\\\nonumber
  & =& \lim_{n \rightarrow \infty} \frac{1}{n} \log ||e^{\hat\Omega_0 \xi_0}\,
  e^{\hat\Omega_1 \xi_1}\,\cdots \,e^{\hat\Omega_{n-1}
    \xi_{n-1}}||
  \\\label{eq:ealph}
& = &\lim_{n \rightarrow \infty} \frac{1}{n} \E \log ||e^{\hat\Omega_0 \xi_0}\,
  e^{\hat\Omega_1 \xi_1}\,\cdots \,e^{\hat\Omega_{n-1}
    \xi_{n-1}}||,
\end{eqnarray}
independently of the matrix norm, as shown in Athreya and Karlin~\cite{ak71}, and in Kingman~\cite{kingm73}.

The parameter $\omega$ may be likened  to a Lyapunov
exponent: given a set of
real matrices $\mathcal{A}=\{A_k:k\geq 1\}$ and a probability distribution ${\p}$ on
$\mathcal{A}$, the maximal Lyapunov exponent $\rho$ for $\aa$ and ${\p}$ is defined to be
\begin{align*}
\rho(\mathcal{A},{\p}) & = \lim_{n \rightarrow \infty} 1/n\, \E\log||A_1\ldots A_n||,
\end{align*}
where $A_k$, $k \geq 1$, are independent and identically distributed random matrices on $\mathcal{A}$ with the distribution $\p$; 
the limit exists and does not depend on the choice of  matrix norm
(Furstenberg and Kesten~\cite{furstenbergandkesten}, 
Oseledec~\cite{oseledec}).  
In our case, $\aa=\{e^{\Omega x}:
\Omega\in \{\Omega_1, \ldots, \Omega_m\}, x \geq 0\}$, the probability distribution
$\p$ is induced by the random environment, and the matrices $A_k$ are
not independent nor identically distributed.  

Lyapunov exponents are hard to compute
(Kingman~\cite{kingm73}, Tsitsiklis and Blondel~\cite{tb97}), except
under very special circumstances, such as in Key~\cite{key87} where the
matrices in the family are assumed to be simultaneously diagonalizable,
or in Lima and Rahibe~\cite{limaandrahibe} where $\mathcal A$ contains two matrices of order 2 only, one of
which is singular.
 For a
thorough survey on the basics of Lyapunov exponents, we refer to
Watkins~\cite{watkins}. 

If $\mathcal A$ is a finite set of nonnegative matrices, Gharavi and
Anantharam~\cite{gharaviandanantharam} give an upper bound for
$\rho(\aa,\p)$  in the form of the maximum of a nonlinear
concave function.
Key~\cite{key90} gives both
upper and lower bounds  determined as follows, on the basis
of (\ref{eq:ealpha}, \ref{eq:ealph}).  Define 
$\sigma_n  =  1/n\, \E\log  ||A_1\cdots A_n|| $ 
and
$\sigma^*_n   =  1/n \,\E\log \{A_1\cdots A_n\}_{jj}$
for some arbitrarily fixed $j$.  One verifies that $\{\sigma_{2^k}\}$
is non-increasing and that $\{\sigma^*_{2^k}\}$ is non-decreasing, so
that these form sequences of upper and lower bounds for $\omega$.
They are, however, not much easier to compute than the Lyapunov
exponent itself.

In the absence of an easily computed exact expression for $\omega$, we look for upper and lower bounds, in  an adaptation of  the approach  developed in \cite{hautphenne2013markovian} for  branching processes subject to binomial catastrophes at random times.  
Our bounds are obtained in three steps.  First, we replace the branching process by a dual, marked Markovian process.  This has the advantage that we deal with the trajectories of a simple continuous-time Markov chain, instead of the conditional expected number of individuals in each type of the original branching process.  Next, we take the expectation with respect to  the $\hat X_k$s and the $\xi_k$s, and obtain an upper bound for $\omega$.  In the third step, we add some information about the history of the dual process and so obtain a lower bound.   
In some cases, we have some leeway in the amount of information that we need to add to obtain a lower bound, and we show that the bound is smaller when  we add more information.  Thus, roughly speaking, we may say that the conditional expectation in (\ref{mp}) is taken with respect to a finely balanced information;  less information leads to an upper bound, while more information yields a lower bound.

Our approach may be used for the analysis of other systems, beyond branching processes.  For instance, Bolthausen and Goldsheid analyse in \cite{bolthausen2000recurrence} a random walk on a strip subject to a random environment and their parameter $\lambda$ in  \cite[Equation 6]{bolthausen2000recurrence} is a Lyapunov exponent.   Such expressions typically arise in the analysis of systems in random environments and should be amenable to an approach similar to ours.


The paper is organised as follows.  The dual process and the marks
are defined in Section~\ref{s:dual} and we determine an
upper and a lower bounds for $\omega$.   
We  discuss in Section~\ref{2du}
an alternative  definition of the dual Markov chain: this second dual yields the same upper bound but a different lower bound.    We give two
illustrative examples in Section~\ref{s:examples}, showing that the ordering of the lower bounds depends on the case.   We
show
in Section~\ref{s:improve} how the lower bound may be increased in the special circumstance where the environmental process $\{X(t)\}$ is cyclic.  Finally,  we prove in 
Section~\ref{s:tightness}     
that our bounds are tight.

\section{Dual process}
   \label{s:dual}

We denote by $Q$ the generator of the environmental process $\{X(t)\}$ and by  $c_\ell = |Q_{\ell\ell}|$, for $1 \leq \ell \leq
m$,  the parameters of the exponential distributions of
sojourn times in the different environmental states.  

   We assume that the matrices $\Omega_\ell$, $1 \leq \ell \leq m$,
   are all irreducible.  Then, each matrix has an eigenvalue
   $\lambda_\ell$ of maximal real part, which is the asymptotic growth
   rate of the population under
   the conditions prevailing while the environmental process is in
   state $\ell$.  The matrices
   $\Omega^*_\ell$ defined by $\Omega^*_\ell=\Omega_\ell-\lambda_\ell
   I$ have each one eigenvalue equal to zero, the other eigenvalues
   having a strictly negative real part.  Furthermore, the
   corresponding left- and right-eigenvectors $\vc u_\ell$ and $\vc
   v_\ell$ are strictly positive and may be normalised by $\vc
   u_\ell\vc1=1$ and $\vc u_\ell\vc v_\ell=1$.

We proceed in two steps.  Firstly, we write (\ref{mp}) as 
\begin{eqnarray}
   \nonumber
\omega &=& \lim_{n \rightarrow \infty}
  \frac{1}{n}[\lambda_1\sum_{k=1}^{n_1}\xi_k^{(1)}+\ldots+\lambda_m\sum_{k=1}^{n_m}\xi_k^{(m)}]  \\
   \label{e:omegA}
  &&+\lim_{n \rightarrow \infty} \frac{1}{n} \log
  \{e^{\hat\Omega^*_0 \xi_0}\, e^{\hat\Omega^*_1
      \xi_1}\,\cdots \,e^{\hat\Omega^*_{n-1}
      \xi_{n-1}}\}_{ij}    \qquad \mbox{a.s.,}
\end{eqnarray}
where $n_\ell=\sum_{k=0}^{n-1}\mathds{1}_{\{\hat{X}_{k}=j\}}$ is the
number of times the environment visits state $\ell$ during the first
$n$ environmental transitions, $\hat\Omega^*_k = \Omega^*_{\hat X_k}$,
and $\xi_k^{(\ell)}$ is the length of the $k$th sojourn interval in
state $\ell$. 

The random variables $\{\xi_k^{(\ell)}, k\geq 1\}$ are independent and
exponentially distributed with parameter $c_\ell$, so that the sum
$\sum_{k=1}^{n_\ell}\xi_k^{(\ell)}/n_\ell$ converges almost surely to
$\E[\xi^{(\ell)}]=1/c_\ell$ by the Strong Law of Large Numbers.  For
the same reason, $n_\ell/n \rightarrow\hat\pi_\ell$ with probability
one as $n\rightarrow\infty$, for all $\ell$, where $\hat{\vc \pi}$ is
the stationary distribution of the jump chain $\{\hat{X}_k\}$.
Therefore, we may write (\ref{e:omegA}) as
\begin{equation}
   \label{e:omegB}
\omega = (\lambda_1\hat{\pi}_1
    /c_1+\ldots+\lambda_m\hat{\pi}_m/c_m)+\Psi  
\end{equation}
where $\Psi$ is a constant such that 
\begin{eqnarray}\label{psi1}\Psi=\lim_{n \rightarrow \infty}
  \frac{1}{n} \log \{e^{\hat\Omega^*_0 \xi_0}\, e^{\hat\Omega^*_1
      \xi_1}\,\cdots \,e^{\hat\Omega^*_{n-1}
      \xi_{n-1}}\}_{ij}   \qquad \mbox{a.s.}
\end{eqnarray}
In the second step, we define $\Delta_\ell=\diag(\vc
v_\ell)$ and 
\begin{equation}
   \label{theta}
\Theta_\ell=\Delta_\ell^{-1}\,\Omega^*_\ell \,\Delta_\ell, \qquad
\mbox{for $1\leq \ell\leq m$.} 
\end{equation}
It is easy to verify for each $\ell$ that $\Theta_\ell$ is a
generator: it has nonnegative off-diagonal elements and the row sums
are equal to zero, so that the diagonal elements are strictly negative.
We have 
$e^{\Omega^*_\ell}=\Delta_\ell \, e^{\Theta_\ell} \, \Delta_\ell^{-1}$,
so that
\begin{equation}
   \label{psi2}
\Psi  =  \lim_{n \rightarrow \infty} \frac{1}{n} \log \{
\hat\Delta_0  \,  e^{\hat\Theta_0 \xi_0} \, \hat\Delta_0^{-1} \,  \hat\Delta_1 \,
e^{\hat\Theta_1 \xi_1} \, \hat\Delta_1^{-1} \, \cdots 
\, \hat\Delta_{n-1} \, e^{\hat\Theta_{n-1} \xi_{n-1}} \, \hat\Delta_{n-1}^{-1}
\}_{ij}
\end{equation}
almost surely, where $\hat\Delta_k = \Delta_{\hat X_k}$ and $ \hat\Theta_k =
\Theta_{\hat X_k}$.

While the probabilistic interpretation of the matrix product in
\eqref{psi1} is not obvious, we can easily give one to the equivalent
random matrix product in~\eqref{psi2}:   we replace the whole
branching process by a two-dimensional Markov chain $\{(X(t),
\varphi(t)): t \geq 0\}$ on the state space $\{1 \ldots m\} \times \{1
\ldots r\}$, with generator 
\[
Q_\Theta = Q \otimes I_r + 
\begin{bmatrix}
\Theta_1 \\ & \Theta_2 \\ & & \ddots \\ & & & \Theta_m
\end{bmatrix},
\]where $I_r$ is the identity matrix of size $r$ (we indicate the size of $I$ when it is not clear by the context).
Like before, $\{X(t)\}$ is a Markov chain with generator $Q$ and the
dual process $\{\varphi(t)\}$ evolves according to the generator $\Theta_\ell$ as
long as $X(\cdot)$ remains equal to $\ell$.  We define as follows the epochs
$\{\tau_k : k = 0, 1, \ldots\}$ of transition for the component
$X(\cdot)$ of the process:
\[
\tau_0 = 0, \qquad \qquad \tau_{k+1} = \tau_k + \xi_k, \quad k \geq 0,
\]
and we define the process $\{\varphi_k, k\in\mathbb{N}\}$  with $\varphi_k=\varphi(\tau_k)$,
\emph{embedded} at the jump epochs for $\{X(t)\}$.
Finally, we associate a sequence $\{Z_k : k = 0, 1, 2, \ldots\}$ of marks to the intervals
$[\tau_k, \tau_{k+1})$, with
\begin{equation}
   \label{e:markz}
Z_k=
(\hat\Delta_k)_{\varphi_k}  (\hat\Delta_k^{-1})_{\varphi_{k+1}} 
=
\dfrac{
  (\vc v_{\hat{X}_k})_{\varphi_k}
}{
(\vc  v_{\hat{X}_k})_{\varphi_{k+1}}
},
\end{equation}
for $k \geq 0$.

\begin{lem}
   \label{t:markproduct}
The parameter $\omega$ may be written as
\begin{equation}
   \label{e:omegC}
\omega = \hat{\vc\pi} C^{-1} \vc\lambda + \Psi
\end{equation}
where $\hat{\vc\pi}$ is the stationary probability vector of the 
environmental jump chain $\{\hat X_k\}$,  $C= \diag(c_1, \ldots, c_m)$, 
$\vc\lambda = (\lambda_1, \ldots , \lambda_m)^\top$, 
and
\begin{equation}
   \label{psi3}
\Psi   =   \lim_{n \rightarrow \infty} \frac{1}{n} \log
\E[Z_0\,Z_1\,\cdots\,Z_{n-1}\,\mathds{1}_{\{\varphi_n=j\}}\,
   |\,\varphi_{0}=i,\vc\xi^{(n)},\hat{\vc X}^{(n)}]  \qquad \mbox{a.s.},
\end{equation}
with $\vc\xi^{(n)}  = (\xi_0, \ldots, \xi_{n-1})$, and $\hat{\vc X}^{(n)}=(\hat
X_0, \ldots, \hat X_{n-1})$.
\end{lem}
\begin{proof}
Equation   (\ref{e:omegC}) is merely (\ref{e:omegB}) written  in a more compact manner.
Furthermore, one has 
\[
\E[Z_k\,\mathds{1}_{\{\varphi_{k+1}=j\}}\,|\,\varphi_{k}=i,\xi_k,\hat{X}_k]
= (\hat\Delta_k e^{\hat\Theta_k \xi_k} \hat\Delta_k^{-1})_{ij},\quad k\geq 0,
\]
and a simple calculation shows that (\ref{psi2}) and (\ref{psi3}) are equivalent.
\end{proof}

We recognise in the first term of (\ref{e:omegC}) the expected
long-term growth rate of the population, while  the second term reflects the fact
that changes in the environment influence all individuals
simultaneously.  The advantage of (\ref{psi3}) over (\ref{psi2}) is
that we now deal with a product of scalar random variables instead of
a product of random matrices.  In the sequel, we shall occasionally
use the notation 
\begin{equation}
   \label{e:zz}
\zz_{n,j} = Z_0 Z_1 \cdots Z_{n-1} \mathds{1}_{\{\varphi_n=j\}}
\end{equation}
in order to simplify the presentation of a few arguments.

Now, we proceed like in~\cite{hautphenne2013markovian}: 
we condition on less information  to find an upper bound,
and we condition on more information to find a lower bound.

\subsection{Upper bound}
   \label{s:upper}
 
To simplify the notation, we define the matrices 
\begin{equation}
   \label{e:Ml}
M_\ell= c_\ell\,\Delta_\ell (c_\ell I -\Theta_\ell)^{-1}\,\Delta_\ell^{-1} = c_\ell[(c_\ell+\lambda_\ell)I-\Omega_\ell]^{-1},
\end{equation}
for $\ell=1,2,\ldots, m$, and the matrix 
\begin{equation}
   \label{e:M}
M = \begin{bmatrix}
M_1 \\ & M_2 \\ & & \ddots \\ & & & M_m
\end{bmatrix}
\end{equation}
of order $rm$.  We also define the transition matrix 
$\hat P = I + C^{-1} Q$ of the jump chain
$\{\hat{X}_k\}$.

\begin{theo} 
   \label{t:omegaU}
An upper bound for 
   $\omega$ is given by  $\omega_U$ with
\[
\omega_U=
   \hat{\vc\pi} C^{-1} \vc\lambda + \log \mathrm{sp}[M(\hat{P}\otimes I_r)].
\]
\end{theo}
\begin{proof}
From (\ref{psi3}, \ref{e:zz}), we have 
$  
e^{\Psi}=\lim_{n\rightarrow\infty}
\E[\zz_{n,j}\,|\,\varphi_{0}=i,\vc\xi^{(n)},\hat{\vc
  X}^{(n)}]^{1/n}
$
a.s.
Taking on both sides expectations with
respect to $\vc\xi^{(n)}$ and $\hat{\vc X}^{(n)}$, we obtain
\begin{eqnarray*}e^\Psi&=&\E\left[\lim_{n\rightarrow\infty}
    \E[\zz_{n,j}\,|\,\varphi_{0}=i,\vc\xi^{(n)},\hat{\vc X}^{(n)}]^{1/n}\right]
  \\&\leq& \lim_{n\rightarrow\infty} \E
  \left[\E[\zz_{n,j}\,|\,\varphi_{0}=i,\vc\xi^{(n)},\hat{\vc X}^{(n)}]^{1/n}\right]\\&\leq&
  \lim_{n\rightarrow\infty} \E\left[
    \E[\zz_{n,j}\,|\,\varphi_{0}=i,\vc\xi^{(n)},\hat{\vc X}^{(n)}]\right]^{1/n}\\&=&\lim_{n\rightarrow\infty}
  \E[\zz_{n,j}\,|\,\varphi_{0}=i]^{1/n},
 \end{eqnarray*} 
where the first inequality follows from Fatou's Lemma, and the second follows
from Jensen's Inequality for concave functions.
 
By conditioning on $\hat{X}_0$ and $\xi_0$, we have
\begin{eqnarray*}\E[Z_0\,\mathds{1}_{\{\varphi_1=j\}}\,|\,\varphi_{0}=i]&=&\sum_{1\leq\ell\leq
    m} \alpha_\ell \int_0^\infty \left[\Delta_\ell e^{\Theta_\ell t}\,\Delta_\ell^{-1}
  \right]_{ij} \,c_\ell e^{-c_\ell t}\,dt\\&=&\sum_{1\leq\ell\leq m}
  \alpha_\ell \left[c_\ell\,\Delta_\ell(c_\ell I -\Theta_\ell)^{-1}\,\Delta_\ell^{-1} 
  \right]_{ij}  \\
 & = & \left[
  (\vc \alpha \otimes I_r) M (\vc 1\otimes I_r)\right]_{ij}
\end{eqnarray*} 
where $\vc\alpha$ is the initial probability vector of $\{X(t)\}$, and the vector
$\vc 1$ is of size~$m$.
By induction, one shows that 
\[
\E[ \zz_{n,j} \,|\,\varphi_{0}=i]=\left[
  (\vc \alpha \otimes I_r)[M(\hat{P}\otimes I_r)]^n(\vc 1\otimes
  I_r)\right]_{ij},
\]
so that 
\begin{eqnarray*}
e^\Psi & \leq & \lim_{n\rightarrow\infty}
\left[ (\vc \alpha \otimes I_r)[M(\hat{P}\otimes I_r)]^n(\vc 1\otimes
  I_r)\right]_{ij}^{1/n}\\ 
&=& \textrm{sp}[M(\hat{P}\otimes I_r)],
\end{eqnarray*} 
independently of $\vc\alpha$, $i$, and $j$.
\end{proof}

\subsection{Lower bound}
   \label{s:lower}

   To obtain a lower bound, we reverse the argument of
   Theorem~\ref{t:omegaU}: we start from a conditional expectation
   given {\em more} information.  To make use of the definition
   (\ref{e:markz}) of the marks, we need to consider two discrete-time Markov chains
   $\{Y_k^{(1)}\}$ and $\{Y_k^{(2)}\}$ embedded at the epochs
   $\{\tau_k\}$, with $Y_k^{(1)} =(\hat X_k, \varphi_k)$ and
   $Y_k^{(2)} =(\hat X_k, \varphi_{k+1})$.    For $\{Y_k^{(1)}\}$, a
   transition from $(\ell,i)$ to
$(\ell',j)$  has probability 
\[
P^{(1)}_{(\ell,i),(\ell',j)}=\int_0^\infty
\left(e^{\Theta_\ell t}\right)_{ij} c_\ell e^{-c_\ell
  t}dt \, \hat{P}_{\ell\ell'}= (N_\ell)_{ij} \hat{P}_{\ell\ell'}
\]  
where 
\begin{equation}
   \label{Ni}
N_\ell = c_\ell(c_\ell I-\Theta_\ell)^{-1} = \Delta_\ell^{-1}
M_\ell \Delta_\ell.
\end{equation}
A similar calculation shows that 
$P^{(2)}_{(\ell,i),(\ell',j)} = \hat{P}_{\ell\ell'}  (N_{\ell'})_{ij} $
and we write
\[
P^{(1)}=N(\hat{P}\otimes I_r),\qquad P^{(2)}=(\hat{P}\otimes I_r)N,
\]
with
\[
N = \begin{bmatrix}
N_1 \\ & N_2 \\ & & \ddots \\ & & & N_m
\end{bmatrix}.
\]

\begin{theo} 
   \label{t:omegal}
A lower bound for $\omega$ is given by $\omega_L$ with 
\begin{equation}
   \label{e:omegal}
\omega_L= \hat{\vc\pi} C^{-1} \vc\lambda + \vc\pi^{(1)}(I - N) \log \bar{\vc v},   
\end{equation} 
where $\vc\pi^{(1)}$ is the stationary probability vector of $P^{(1)}$
and 
\[
\bar{\vc v}=\begin{bmatrix} \vc v_1 \\\vc v_2 \\ \vdots \\ \vc v_m \end{bmatrix}.
\]
\end{theo}
\begin{proof}
We start from the conditional expectation of the product of marks,
given $\vc\varphi^{(n)} = (\varphi_1, \ldots, \varphi_{n-1})$ in addition to $\vc\xi^{(n)}$
and $\hat{\vc X}^{(n)}$.   We readily find from (\ref{e:markz}) that 
\begin{eqnarray*} \lefteqn{
    \E[ \zz_{n,j} \,|\,\varphi_0=i,\vc\varphi^{(n)},
    \vc\xi^{(n)}, \hat{\vc X}^{(n)}]}  \\
  &=& \dfrac{(\vc v_{\hat{X}_0})_{i}}{(\vc
    v_{\hat{X}_0})_{\varphi_{1}}}\,\dfrac{(\vc
    v_{\hat{X}_1})_{\varphi_1}}{(\vc
    v_{\hat{X}_1})_{\varphi_{2}}}\,\cdots \,\dfrac{(\vc
    v_{\hat{X}_{n-2}})_{\varphi_{n-2}}}{(\vc
    v_{\hat{X}_{n-2}})_{\varphi_{n-1}}}\,\dfrac{(\vc
    v_{\hat{X}_{n-1}})_{\varphi_{n-1}}}{(\vc v_{\hat{X}_{n-1}})_{j}}
  \left(e^{\hat\Theta_{n-1}\xi_{n-1}}\right)_{\varphi_{n-1},j}
  \\
 & = & 
\prod_{\ell=1}^m \prod_{a=1}^r (\vc
v_\ell)_a^{n^{(1)}_{(\ell,a)}-n^{(2)}_{(\ell,a)}}
\dfrac{(\vc v_{\hat{X}_0})_{i}}{(\vc
  v_{\hat{X}_{n-1}})_{j}}\left(e^{\hat\Theta_{n-1}\xi_{n-1}}\right)_{\varphi_{n-1},j},
\end{eqnarray*} 
where 
\[
n^{(1)}_{(\ell,a)}=\sum_{k=1}^{n-1}
\mathds{1}_{\{Y_k^{(1)}= (\ell,a)\}}
\qquad \mbox{and} \qquad 
n^{(2)}_{(\ell,a)}=\sum_{k=0}^{n-2}
\mathds{1}_{\{Y_k^{(2)}=(\ell,a)\}}.
\]
Therefore,
\begin{align}
  \nonumber \lim_{n \rightarrow \infty} &
  \E[ \zz_{n,j} \,|\,\varphi_0=i,\vc\varphi^{(n)},
  \vc\xi^{(n)}, \hat{\vc X}^{(n)}]^{1/n} 
\\ \nonumber 
&=
  \lim_{n\rightarrow \infty} \prod_{\ell=1}^m \prod_{a=1}^r (\vc
  v_\ell)_a^{\frac{n^{(1)}_{(\ell,a)}-n^{(2)}_{(\ell,a)}}{n}}
  \left(\dfrac{(\vc v_{\hat{X}_0})_{i}}{(\vc
      v_{\hat{X}_{n-1}})_{j}}\left(e^{\hat\Theta_{n-1}\xi_{n-1}}\right)_{\varphi_{n-1},j}\right)^{1/n}
\\   \label{prod}
    &=
 \prod_{\ell=1}^m \prod_{a=1}^r (\vc
 v_\ell)_a^{\pi^{(1)}_{(\ell,a)}-\pi^{(2)}_{(\ell,a)}}
\qquad \mbox{a.s.}
\end{align}
where $\vc \pi^{(1)}$ and $\vc\pi^{(2)}$ are the
stationary distribution vectors of the stochastic matrices $P^{(1)}$
and $P^{(2)}$.
Now,
\begin{align}
   \nonumber
e^{\Psi} & = 
\lim_{n\rightarrow\infty}
\E[ \zz_{n,j} \,|\,\varphi_0=i,
  \vc\xi^{(n)}, \hat{\vc X}^{(n)}]^{1/n}
\\
   \nonumber
&= 
\lim_{n\rightarrow\infty} \E_\varphi \left[
  \E[ \zz_{n,j} \,|\,\varphi_0=i,\vc\varphi^{(n)},
  \vc\xi^{(n)}, \hat{\vc X}^{(n)}]\right]^{1/n},
\intertext{where the outermost expectation is with respect to
  ${\vc\varphi}^{(n)}$,}
   \nonumber
& \geq \lim_{n\rightarrow\infty} \E_\varphi \left[
 \E[ \zz_{n,j} \,|\,\varphi_0=i,\vc\varphi^{(n)},
  \vc\xi^{(n)}, \hat{\vc X}^{(n)}]^{1/n}\right]
\\
   \label{e:espz}
& \geq \E_\varphi \left[\lim_{n\rightarrow \infty}
  \E[ \zz_{n,j} \,|\,\varphi_0=i,\vc\varphi^{(n)},
  \vc\xi^{(n)}, \hat{\vc X}^{(n)}]^{1/n}\right]
\end{align}
where we use Jensen's Inequality followed by Fatou's Lemma.

From (\ref{e:omegC}),  (\ref{prod}) and (\ref{e:espz}) , we find that
\[
\omega_L= \hat{\vc\pi} C^{-1} \vc\lambda + (\vc\pi^{(1)} - \vc\pi^{(2)}) \log \bar{\vc v}.   
\]
Finally, it is easy to verify that $\vc\pi^{(2)} = \vc\pi^{(1)} N$,
which concludes the proof.
\end{proof}

\section{Alternative dual process}\label{2du}

Instead of using the right-eigenvectors 
of the matrices $\Omega_\ell$
like in (\ref{theta}),
one may set up another dual process, starting from the
left-eigenvectors.
Define $\Delta'_\ell = \diag(\vc u_\ell)$ and 
\[
\Theta'_\ell =(\Delta'_\ell)^{-1}\,(\Omega^*_\ell)^T\,\Delta'_\ell
\qquad \mbox{for $1\leq \ell\leq m$.}
\]
With these, one shows that 
\begin{eqnarray}
\nonumber
\Psi&=&\lim_{n \rightarrow \infty} \frac{1}{n} \log
\left\{\bar\Delta_{n-1} \, e^{\bar\Theta_{n-1}
      \xi_{n-1}} \, \bar\Delta_{n-1}^{-1}\,
    \bar\Delta_{n-2} \, e^{\bar\Theta_{n-2}
      \xi_{n-2}} \, \bar\Delta_{n-2}^{-1}\,\cdots\right. 
\\
   \label{e:psialt}
&&\left.\quad\phantom{\lim_{n \rightarrow \infty} \frac{1}{n} \log}
  \,\bar\Delta_{0} \, e^{\bar\Theta_{0} \xi_{0}} \, 
   \bar\Delta_{0}^{-1}\right\}_{ji},   \qquad \mbox{a.s.}
\end{eqnarray} 
where $\bar\Delta_k = \Delta'_{\hat X_k}$ and
$\bar\Theta_k = \Theta'_{\hat X_k}$.  
We define $\{\widetilde X_k\}$ as the \emph{time-reversed} version of
the environmental jump chain, with transition matrix 
\[
\widetilde P = \diag(\hat{\vc\pi})^{-1}  \hat P^T \diag(\hat{\vc\pi}),
\]
and we rewrite (\ref{e:psialt}) as
\begin{eqnarray}
   \label{e:psialtB}
\lefteqn{\Psi =} \\
 &&
  \nonumber
\lim_{n \rightarrow \infty} \frac{1}{n} \log
\left\{\widetilde\Delta_0 \, e^{\widetilde\Theta_0
      \eta_0} \, \widetilde\Delta_0^{-1}\,
    \widetilde\Delta_1 \, e^{\widetilde\Theta_1
      \eta_1} \, \widetilde\Delta_1^{-1}\,\cdots
  \,\widetilde\Delta_{n-1} \, e^{\widetilde\Theta_{n-1} \eta_{n-1}} \, 
   \widetilde\Delta_{n-1}^{-1}\right\}_{ji}, 
\end{eqnarray} 
with probability one, 
where $\widetilde\Delta_k = \Delta'_{\widetilde X_k}$, 
$\widetilde\Theta_k = \Theta'_{\widetilde X_k}$ and $\eta_k$ is
exponential with parameter $c_{\widetilde X_k}$.  

We follow the same steps as in Section \ref{s:dual} and obtain a new
lower bound; we omit the proof here.

\begin{theo}
   \label{t:boundsB}
An alternative lower bound for $\omega$ is given by
\[
\widetilde \omega_L = \hat{\vc\pi} C^{-1} \vc\lambda + \widetilde{\vc\pi}^{(1)} (I-
\widetilde N) \log \bar{\vc u}^T,
\]
where 
\begin{align*}
\widetilde  N & = \begin{bmatrix}
\widetilde  N_1 \\ & \widetilde  N_2 \\ & & \ddots \\ & & &\widetilde  N_m
\end{bmatrix},
\\
\widetilde  N_\ell  & = c_\ell(c_\ell I-\Theta'_\ell)^{-1},
\end{align*}
the vector $\widetilde {\vc\pi}^{(1)}$ is the stationary probability vector of 
$
\widetilde  P^{(1)}=\widetilde  N(\widetilde {P}\otimes I_r)
$
and 
\[
\bar{\vc u} = \begin{bmatrix}  \vc u_1 & \vc u_2 & \ldots &\vc u_m \end{bmatrix}.
\]
\qed
\end{theo}
If we repeat from (\ref{e:psialtB}) the argument from
Theorem~\ref{t:omegaU}, we find that $\omega \leq \widetilde
\omega_U$, with
\[
\widetilde \omega_U = \hat{\vc\pi} C^{-1} \vc\lambda + \log \mathrm{sp}[M^T
(\widetilde{P}\otimes I_r)].
\]
This is not a new upper bound, however, as $\omega_U = \widetilde
\omega_U$.  Indeed,
\begin{align*}
\sp [M^T (\widetilde P \otimes I_r)] 
& =  \sp [M^T (\Delta_\pi^{-1} \otimes I_r ) (\hat P^T \otimes
I_r) (\Delta_\pi \otimes I_r )] , \quad \mbox{with $\Delta_\pi =
  \diag(\hat{\vc\pi})$} 
\\
 & = \sp [(\Delta_\pi \otimes I_r )  M^T (\Delta_\pi^{-1} \otimes I_r ) (\hat P^T \otimes
I_r) ] , 
\end{align*}
as the spectral radius of a product of matrices is invariant under a
cyclic permutation of the factors.  The first three factors are
block-diagonal matrices and a simple calculation shows that they
commute, so that 
\begin{align*}
  \sp [M^T (\widetilde P \otimes I_r)] & = \sp [ M^T (\hat P^T
  \otimes I_r) ] \\
 & = \sp [(\hat P \otimes I_r) M] = \sp[M (\hat P\otimes I_r)].
\end{align*}
The lower bounds, on the other hand, are obtained from different vectors, and
numerical experimentation has shown that they may indeed be very
different, without one being generally closer to $\omega$.  This is
illustrated in the next section.

\section{Numerical examples} 
   \label{s:examples}
   In this section, we illustrate our bounds with help of two
   examples. The value of $\omega$ as defined in \eqref{mp} is
   approximated by simulating $2000$ paths of the branching process
   for $n$ environmental transitions, with $n$ taking values up to
   $2000$; we denote this approximation by $\omega_{sim}$.

\begin{ex}[A two-state random environment]\label{ex1}
  In our first example, $m=2$ and the generator of $\{X(t)\}$ is
\[
Q= \left[\begin{array}{cc} -5  &  5\\
     2   &-2\end{array}\right].
\]
Its stationary
 distribution is $\vc\pi=[0.2857,0.7143]$. 
 We take  $r=2$  and
\begin{equation}
  \label{omega1}
\Omega_{1}=\left[\begin{array}{cc} -15  &  12\\
     9   &-29\end{array}\right],\quad  \Omega_{2}=\left[\begin{array}{cc} -13 &   16\\
     23  & -12\end{array}\right].
\end{equation}
 The dominant eigenvalues of $\Omega_{1}$ and $\Omega_{2}$ are
 $\lambda_1=-9.47$ and $\lambda_2=6.69$ respectively, so the branching
 process is subcritical in state 1 and supercritical in state 2.   As
 $\pi_2 > \pi_1$, we expect the whole process to be supercritical and
 this is confirmed by our results, summarised in the table below and
 in Figure~\ref{fi1}.
 
      \begin{figure}[t] \begin{center}\includegraphics[angle=0,
 width=10cm]{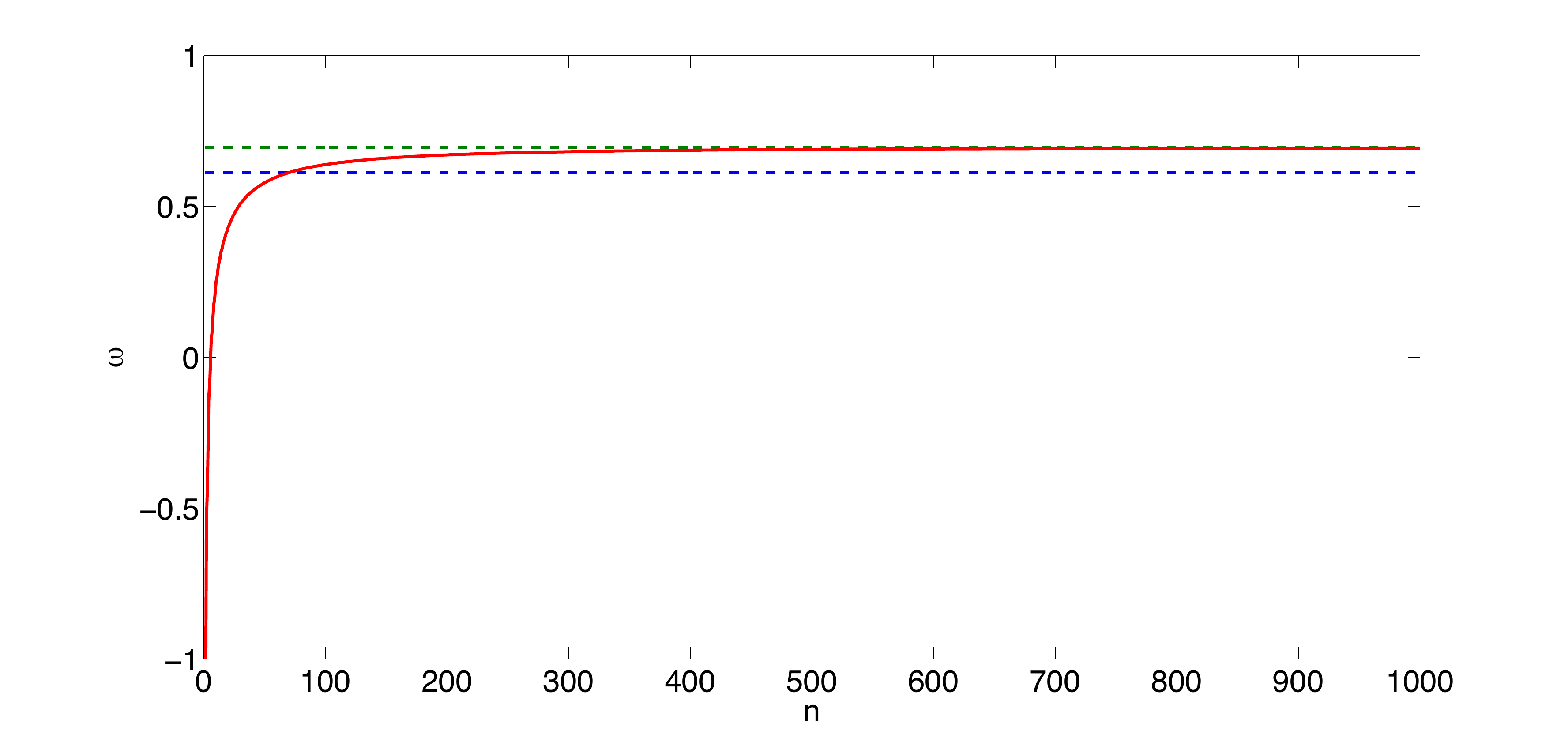}\end{center} \caption{\label{fi1}Plain line: simulations of $\omega$ (Cesaro sequence) for successive values of $n$; lower dashed line: $\omega_L$; upper dashed line: $\omega_U$.}
 \end{figure}\end{ex}

 \begin{center}
   \begin{tabular}{cc|c|c}
    $\omega_L $ &  $\widetilde\omega_L$ & $\omega_{sim}$ & $\omega_U$ \\
\hline
  0.6107 & 0.6867 & 0.6933  & 0.6964
   \end{tabular}
 \end{center}
The simulation results are presented
in Figure \ref{fi1}: the plain line is the average, over 2000 simulated
paths, of the right-hand side of (\ref{mp}), one does see its
convergence as $n$ increases.  The upper and lower dashed lines are
for $\omega_U$ and $\omega_L$.  Clearly, $\omega$ is positive.  For
this example, $\omega_L $ is not very good and $\widetilde\omega_L
$ gives a better lower bound.

 \begin{ex}[A three-state random environment]\label{ex2}
 
 We add a third state to the random environment and we assume the following generator
 \[
Q= \left[\begin{array}{ccc} -4 &2&2\\
    1&-1&0\\
    2&4&-6\end{array}\right],
\]
with stationary distribution $\vc\pi=[0.2143,0.7143,0.0714]$.  The matrices
   $\Omega_{1}$ and $\Omega_2$ are the same as in Example \ref{ex1},
   and the third is
\begin{equation}
   \label{omega3}
 \Omega_3=\left[\begin{array}{cc} -39   & 12\\3 &
     -17\end{array}\right]
\end{equation}
with $\lambda_3=-15.47$, so the branching process is very subcritical
in the new environmental state.   In this case, we obtain the
following values.
 \begin{center}
   \begin{tabular}{cc|c|c}
    $\widetilde\omega_L $ &  $\omega_L$ & $\omega_{sim}$ & $\omega_U$ \\
\hline
 0.5725 & 0.6442 & 0.7181 & 0.7475
   \end{tabular}
 \end{center}
We represent the bounds and the simulation in Figure \ref{fi2}. 
For this example, the better lower bound is $\omega_L$.
    
    \begin{figure}[t] \begin{center}\includegraphics[angle=0,
 width=10cm]{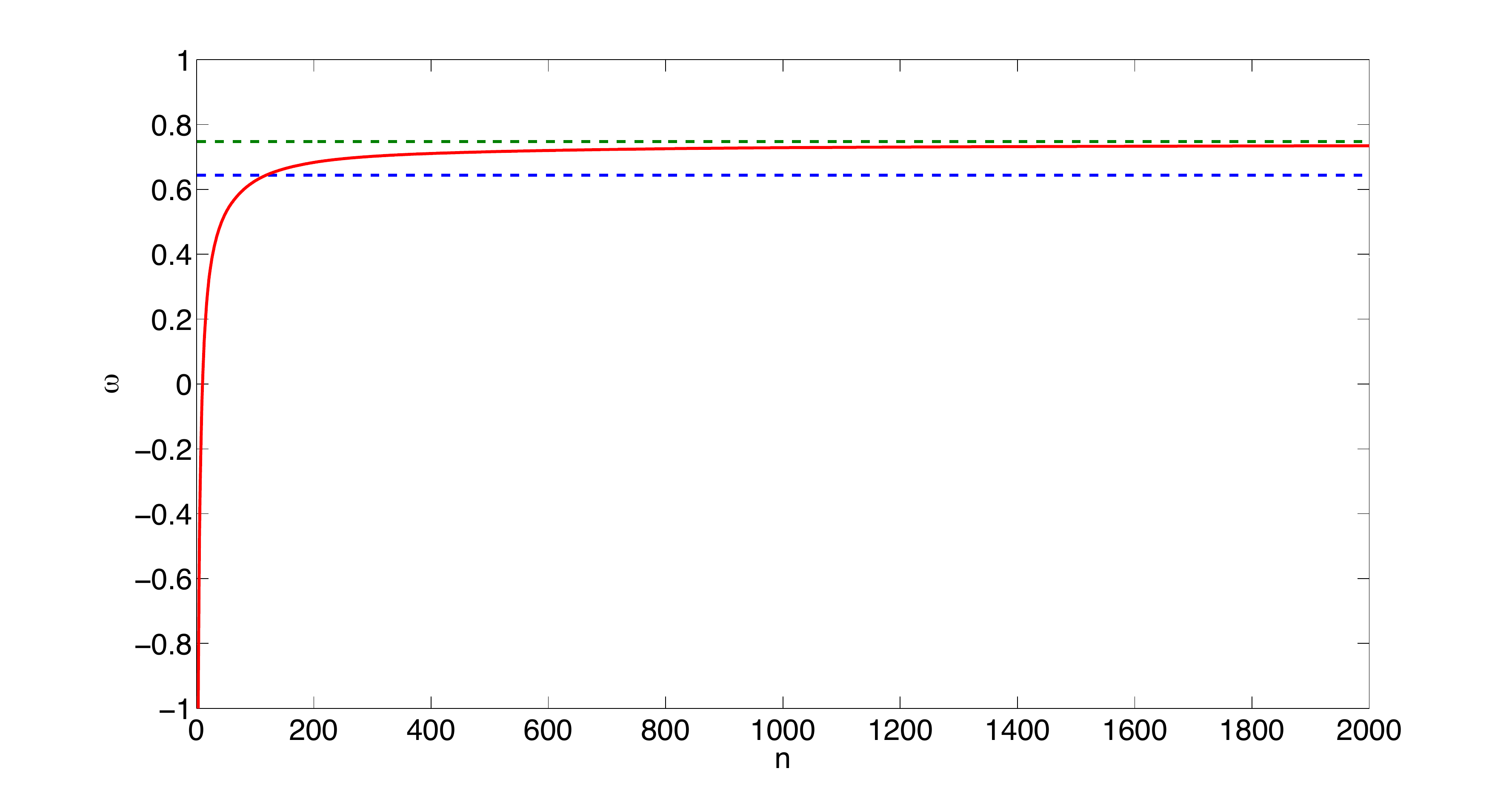}\end{center} \caption{\label{fi2}Plain line: simulations of $\omega$ (Cesaro sequence) for successive values of $n$; lower dashed line: $\omega_L$; upper dashed line: $\omega_U$.}
 \end{figure}\end{ex}
 
The examples have been chosen so that the environment spends
asymptotically the same amount of time in the supercritical state 2
(the stationary probability is $\pi_2=0.7143$ in both cases).  In
addition, state 3 is more subcritical than state 1.  Nevertheless,
$\omega$ is greater, the branching process is overall more
supercritical, in the second example.  It is worth mentioning that we
have generally observed in our experimentation that the upper bound is
more accurate than the lower bounds.

\section{Adding less information}
   \label{s:improve}

The lower bound $\omega_L$ is obtained by adding information in the
form of  \textit{all} successive states of the embedded dual process
$\{\varphi_k\}$.
In general, it seems that $\omega_L$ is not a very close bound for
$\omega$, not as close as $\omega_U$ at any rate, and neither is
$\widetilde\omega_L$.   We attempt, therefore, to obtain a better
lower bound by adding less information.  This is doable
in the special case of a cyclic environmental process.

Assume that the jump chain $\{\hat X_k\}$ follows a deterministic
cyclic path, and assume, without loss of generality, that $\hat X_k =k\mod m +1$.  In this case, (\ref{psi3}) is equivalent to
\begin{align}
   \nonumber
\Psi & = \lim_{n \rightarrow \infty} \frac{1}{n} 
\log \E[ \zz_{n,j} \,
   |\,\varphi_{0}=i,\vc\xi^{(n)}]  \qquad \mbox{a.s.} \\
   \label{e:psiNrho}
 & = \lim_{n \rightarrow \infty} \frac{1}{nm} 
\log \E_\rho [ \E [ \zz_{nm,j} \,
   |\,\varphi_{0}=i,\vc\xi^{(nm)}, \vc\rho^{(n)}] ]  \qquad \mbox{a.s.} 
\end{align}
where the outermost expectation is with respect to $\vc\rho^{(n)}$ defined as $\vc\rho^{(n)} =  
(\varphi_m, \ldots, \varphi_{(n-1)m})$.  That
is, in contrast with Theorem \ref{t:omegal}, we do not condition on
the whole sequence of states of the dual at the epochs $\{\tau_k\}$ but only at
the beginning of cycles for $\{\hat X_k\}$.

We redefine as follows the products $\zz_{n,j}$ in (\ref{e:markz},
\ref{e:zz}):
\[
\zz_{nm,j} = v_{m, \varphi_0} R_0 R_1 \cdots R_{n-1}
v_{m,\varphi_{nm}}^{-1}  \mathds{1}_{\{\varphi_{nm}=j\}},
\]
where
\begin{equation}
   \label{e:Rk}
R_k = \frac{
v_{1,\varphi_{km}} v_{2,\varphi_{km+1}}  \cdots  v_{m,\varphi_{km+m-1}}
}{
v_{m,\varphi_{km}} v_{1,\varphi_{km+1}}  \cdots    v_{m-1,\varphi_{km+m-1}}
}
\end{equation}
and we use the fact that $(R_0, R_1, \ldots, R_{k-1})$ is conditionally
independent of $(R_k, R_{k+1}, \ldots)$, given $\varphi_{km}$.  Thus,
\begin{align}\nonumber
\E [ \zz_{nm,j} \, 
   |\,\varphi_{0}=i, &\vc\xi^{(nm)}, \vc\rho^{(n)}]  \\\nonumber
& = v_{m,i} \prod_{k=0}^{n-2} \E[R_k | \,\varphi_{0}=i,\vc\xi^{(nm)},
\vc\rho^{(n)} ]   
\\ \nonumber
& \qquad \E[ R_{n-1}  v_{m,\varphi_{nm}}^{-1}
\mathds{1}_{\{\varphi_{nm}=j\}} | \,\varphi_{0}=i,\vc\xi^{(nm)}, \vc\rho^{(n)}]
\intertext{which we rewrite as}\nonumber
& = v_{m,i}   \E[R_0 | \,\varphi_{0}=i, \vc\zeta_0, \varphi_m ]     
  \prod_{k=1}^{n-2} \E[R_k | \,\varphi_{km}, \varphi_{(k+1)m}, \vc\zeta_k
  ]   \\\nonumber
 & \qquad \E[ R_{n-1}  v_{m,\varphi_{nm}}^{-1} 
\mathds{1}_{\{\varphi_{nm}=j\}} | \,\varphi_{(n-1)m}, \vc\zeta_{n-1}]
\intertext{with $\vc\zeta_k = (\xi_{km}, \xi_{km+1}, \ldots,
  \xi_{km+m-1}$)}
  & = f \prod_{k=1}^{n-2} \E[R_k | \,\varphi_{km}, \varphi_{(k+1)m},
  \vc\zeta_k ]\label{prout}
\end{align}
where we collect in 
\[
f = v_{m,i}   \E[R_0 | \,\varphi_{0}=i, \vc\zeta_0, \varphi_m ]     \E[ R_{n-1}  v_{m,\varphi_{nm}}^{-1} 
\mathds{1}_{\{\varphi_{nm}=j\}} | \,\varphi_{(n-1)m}, \vc\zeta_{n-1}]
\]
all the factors which play no role in the limit as $n \rightarrow
\infty$.   Observe that $\vc\zeta_0$, $\vc\zeta_1$, \ldots  \ are independent and identically distributed random
$m$-tuples, with the same distribution as $\vc\zeta^* = (\xi_1^*,
\xi_2^*, \ldots, \xi_m^*)$, where $\xi_1^*$, $\xi_2^*$, \ldots,
$\xi_m^*$ are independent, exponentially distributed random variables,
respectively with parameters $c_1$, $c_2$, \ldots, $c_m$.  Thus, from \eqref{e:psiNrho} and \eqref{prout},
\begin{align}
   \nonumber
e^\Psi & = \lim_{n \rightarrow \infty} \{  \E_\rho[
 f \prod_{k=1}^{n-2} \E[R_k | \,\varphi_{km}, \varphi_{(k+1)m},
  \vc\zeta_k ]   ] \}^{1/nm}
\\
   \nonumber
 & \geq \E_\rho[ \lim_{n \rightarrow \infty} \{  
 f \prod_{k=1}^{n-2} \E[R_k | \,\varphi_{km}, \varphi_{(k+1)m},
  \vc\zeta_k ]   \}^{1/nm}]
\\
   \label{e:inter}
& = \E_\rho[ \lim_{n \rightarrow \infty} \{  
 \prod_{k=1}^{n-2} \E[R_k | \,\varphi_{km}, \varphi_{(k+1)m},
  \vc\zeta_k ]   \}^{1/nm}];
\end{align}
the inequality is justified by Fatou's lemma and Jensen's inequality.

We may now prove the following property.

\begin{theo}
    \label{t:omegastar}
    A lower bound for $\omega$ is given by
    \begin{equation}
       \label{e:star1}
    \omega_L^* = \hat{\vc\pi} C^{-1} \vc\lambda + \frac{1}{m} \sum_{1 \leq i,j \leq r} \beta_{ij} \E[\log \E[R_0| \varphi_0=i, \varphi_m=j, \vc\zeta^*]],
    \end{equation}
    where 
    \begin{equation}
    \label{e:beta}
    \beta_{ij} = \alpha_i (N_1 N_2 \cdots N_m)_{ij},
    \end{equation}
the matrices $N_\ell$ are defined in (\ref{Ni}) and $\vc\alpha$ is the stationary vector of their product:
\begin{equation}
  \label{e:alpha}
  \vc\alpha N_1 N_2 \cdots N_m = \vc\alpha, \qquad \vc\alpha \vone = 1.
\end{equation}
\end{theo}

\begin{proof}
We reorganise the product in (\ref{e:inter}) and group together the
factors with equal values for $\varphi_{km}$ and $\varphi_{(k+1)m}$,
obtaining that 
\begin{equation}
   \label{e:interb}
e^\Psi \geq \E_\rho[ \lim_{n \rightarrow \infty} ( \prod_{1 \leq i, j
  \leq r} \prod_{k=1}^{n_{ij}}  R_{k;i,j}
)^{1/nm}]  
\end{equation}
where
$n_{ij}=\sum_{k=1}^{n-2}\mathds{1}_{\{\varphi_{m k}=i,\varphi_{m(k+1)}=j
  \}} $ and, for fixed $i$ and $j$, $\{R_{1;i,j}, R_{2;i,j}, \ldots \}$
are i.i.d. random variables with the same distribution as $\E[R_0 |
\varphi_0=i, \varphi_m = j, \vc\zeta^*]$.   

By the Strong Law of Large Numbers, the limits $\beta_{ij} = \lim_{n
  \rightarrow \infty} n_{ij}/n$ exist and
\begin{align*}
\beta_{ij} & = \lim_{k \rightarrow \infty} \p[\varphi_{km} = i,
\varphi_{(k+1)m} =j]
\\
  & = \lim_{k \rightarrow \infty} \p[\varphi_{km} = i]
  \p[\varphi_{m} =j| \varphi_0=i]
\\
  & = \alpha_i  (N_1 N_2 \cdots N_m)_{ij}.
\end{align*} 
Further, by the Strong Law of Large Numbers again, for fixed $i$ and
$j$,
\[
\lim_{n \rightarrow \infty}  (\prod_{k=1}^{n_{ij}} 
R_{k;i,j})^{1/n}
 =
\exp ( \beta_{ij} \E[ \log \E[ R_0 | \varphi_0 =i, \varphi_m=j, \vc\zeta^*]  ] )
\]
so that the limit in (\ref{e:interb}) is independent of
$\vc\rho^{(n)}$
and the inequality becomes  $\Psi \geq \Psi^*$, where
\[
\Psi^*
  =\frac{1}{m}  \sum_{1 \leq i, j \leq r} \beta_{ij} \E[ \log \E[
 R_0 | \varphi_0 =i, \varphi_m=j, \vc\zeta^*]  ],
\]
which proves that (\ref{e:star1}) is a lower bound for $\omega$.

Now, 
\begin{align}
   \nonumber
\Psi^* & \geq \frac{1}{m}  \sum_{1 \leq i, j \leq r} \beta_{ij} \E[ \E[ \log
 R_0 | \varphi_0 =i, \varphi_m=j, \vc\zeta^*]  ]
\intertext{by Jensen's inequality}
   \nonumber
& = \frac{1}{m}  \sum_{1 \leq i, j \leq r} \beta_{ij} \E[ \log
 R_0 | \varphi_0 =i, \varphi_m=j]  
\\
   \label{e:psii}
& = \frac{1}{m}  \sum_{1 \leq i \leq r} \alpha_i \E[ \log
 R_0 | \varphi_0 =i].
\end{align}
From (\ref{e:Rk}) we find that
\begin{align*}
\E[ \log
 R_0  | \varphi_0 =i]  ] 
 & = \E[\log v_{1,i} + \sum_{\ell = 2}^m \log
 v_{\ell, \varphi_{\ell-1}} - \log v_{m,i} - \sum_{\ell = 1}^{m-1}
 \log v_{\ell, \varphi_\ell}]
\\
 & = (\sum_{\ell = 1}^m \prod_{t=1}^{\ell-1} N_t \log \vc v_\ell )_i
      - \log v_{m,i}
      - (\sum_{\ell = 1}^{m-1} \prod_{t=1}^{\ell} N_t \log \vc v_\ell
      )_i
\\
 & = (\sum_{\ell = 1}^m \prod_{t=1}^{\ell-1} N_t (I-N_\ell) \log \vc v_\ell )_i
\end{align*}
and so, by (\ref{e:psii}),
\begin{equation}
   \label{e:omegal2}
\Psi^*  \geq \frac{1}{m} \vc\alpha (\sum_{\ell = 1}^m \prod_{t=1}^{\ell-1} N_t (I-N_\ell) \log \vc v_\ell ).
\end{equation}
Using the cyclic structure of $\hat P$, one shows that the expression $\vc\pi^{(1)} (I-N) \log \bar{\vc v}$ defined in Theorem~\ref{t:omegal} is identical to the right-hand side of (\ref{e:omegal2}).  This proves that $\omega_L^* \geq \omega_L$.
\end{proof}

Denote by $A_{ij}$, $1 \leq i, j \leq r$, the expectations in (\ref{e:star1}):
\[
A_{ij} = \E[\log \E[R_0| \varphi_0=i, \varphi_m=j, \vc\zeta^*]].
\]
These need to be determined for $\omega_L^*$ to be of practical use.    One verifies from first principles that 
\[
\E[R_0| \varphi_0, \varphi_m, \vc\zeta^*] = (\prod_{1 \leq k \leq m} \Delta_k^* e^{\Theta_k \xi_k^*})_{\varphi_0,\varphi_m} / (\prod_{1 \leq k \leq m}  e^{\Theta_k \xi_k^*})_{\varphi_0,\varphi_m}
\]
where $\Delta_1^* = \Delta_1 \Delta_m^{-1}$ and  $\Delta_k^* = \Delta_k \Delta_{k-1}^{-1}$, $k=2, \ldots, m$, so that 
\[
A_{ij} = \E[\log (\prod_{1 \leq k \leq m} \Delta_k^* e^{\Theta_k \xi_k^*})_{ij}] - \E[ \log (\prod_{1 \leq k \leq m}  e^{\Theta_k \xi_k^*})_{ij}].
\]
Although we do not have an explicit form for the expectations above, estimations by simulation are easily obtained, and this is what we did to compare $\omega_L^*$ and $\omega_L$ in the two examples below.

\begin{ex}[Continuation of Example \ref{ex1}] 
The random environment of Example \ref{ex1} has only two states and is automatically cyclic.  The new lower bound is $\omega_L^*=    0.6577$ and is indeed larger than $\omega_L=0.6107$. 
\end{ex}
 
 \begin{ex}[A three-state cyclic random environment] 
 This is similar to Example \ref{ex2}: it is a three-state random environment with matrices $\Omega_\ell$ defined in (\ref{omega1}, \ref{omega3}); the cyclic generator is 
$$Q= \left[\begin{array}{ccc} -4 &4&0\\
    0&-1&1\\
    6&0&-6\end{array}\right],$$
and the stationary distribution is $\vc\pi = [0.1765, 0.7059, 0.1176]$.  The bounds and the approximation are given in the table below
 \begin{center}
   \begin{tabular}{ccc|c|c}
   $\widetilde\omega_L $ &  $\omega_L$ &  $\omega_L^*$ & $\omega_{sim}$ & $\omega_U$ \\
\hline
0.1494 & 0.3480 & 0.3906 & 0.4218 & 0.4250
     \end{tabular}
 \end{center}
 \end{ex}
 
 We see that in both examples, the difference $\omega_{sim}-\omega_L^*$  is less than  half the difference $\omega_{sim}-\omega_L$.  Needless to say, we might apply the same procedure to the dual of Section~\ref{2du}, thereby obtaining another lower bound, closer to $\omega$ than $\widetilde\omega_L$. In the same manner as there is no systematic difference between $\widetilde\omega_L$ and $\omega_L$, we do not expect that there would be a systematic preference between this additional bound and $\omega_L^*$.

\section{Tightness of the bounds} 
   \label{s:tightness}

The bounds are tight in that there exist branching processes for which
the bounds are all equal and equal to $\omega$.  We show below that
such is the case for single-type processes ($r=1$), and for processes such that the matrices $\Omega_\ell$ commute for all $\ell$; it is in particular the case for  processes
with constant growth rate ($\Omega_\ell = \Omega$ for all $\ell$, for
some $\Omega$).

\begin{lem} If $r=1$, then
  $\omega_L=\widetilde\omega_L= \omega_L^* =
  \omega_U=\omega= \hat{\vc\pi} C^{-1} \vc\lambda$.
\end{lem}

\begin{proof}
  If $r=1$, all matrices reduce to scalars, so that
  $\lambda_\ell=\Omega_\ell$ and $\Omega_\ell^*=0$ for all $\ell$.
  Therefore, $\Psi=0$ by \eqref{psi1} and $\omega = \hat{\vc\pi} C^{-1}
  \vc\lambda$ by (\ref{e:omegC}).

Furthermore, $M_\ell=1$ for all $\ell$, so that  $M(\hat{P}\otimes
I)=\hat{P}$ is a stochastic matrix and $\log
\textrm{sp}[M(\hat{P}\otimes I)]=0$.  We have thus proved that
$\omega_U = \omega$.

Finally, $\Theta_\ell=0$, so that $N_\ell=1$,  for all $\ell$.
Therefore, $I-N=0$ and we conclude that $\omega_L = \omega$.  The same
argument gives us $\widetilde\omega_L = \omega$.   Since $\omega_L \leq \omega_L^* \leq \omega$, this shows that $\omega_L^* = \omega$ as well.
\end{proof}

\begin{lem} If the matrices $\Omega_\ell$ are mutually commutative for all $1\leq \ell\leq m$, then $\omega_L = \widetilde\omega_L = \omega_L^* = \omega_U=\omega =
  \hat{\vc\pi} C^{-1} \vc\lambda$.
\end{lem} 

\begin{proof}
Once more, it is enough to prove that $\textrm{sp}[M(\hat{P}\otimes I)]=1$ and $\vc\pi^{(1)} (I-N)=0$. Note that since the matrices $\Omega_\ell$ are mutually commutative for all $\ell$, their all share the same strictly positive Perron-Frobenius right-eigenvector, that we denote by $\vc v$. Let $\lambda_\ell$ denote the eigenvalue of maximal real part of $\Omega_\ell$. Then the matrices $M_\ell$ defined in (\ref{e:Ml}) are such that
\begin{align*}
M_\ell \vc v & = c_\ell ((c_\ell + \lambda_\ell) I - \Omega_\ell)^{-1} \vc v  \\
  & = c_\ell (c_\ell + \lambda_\ell - \lambda_\ell)^{-1} \vc v \\
  & = \vc v.
\end{align*}
Therefore, 
\[
M (\hat P \otimes I) (\vone \otimes \vc v) = M (\vone
\otimes \vc v) = \vone \otimes \vc v,
\]
and since $\vone \otimes \vc v > \vc 0$, this proves that
$\textrm{sp}[M(\hat{P}\otimes I)]=1$ by the Perron-Frobenius Theorem.

Second, the mutual commutativity of the matrices $\Omega_\ell$ implies the mutual commutitivity of the matrices $\Theta_\ell$, and we denote by $\vc w$ the common strictly positive left-eigenvector of the matrices $\Theta_\ell$ associated to their dominant eigenvalue 0, normalized such that $\vc w\vc 1=1$. The matrices $N_\ell$ defined in
(\ref{Ni}) are then such that
\[
\vc w N_\ell = c_\ell \vc w (c_\ell I - \Theta_\ell)^{-1} = \vc w
\]
for all $\ell$, and so
\[
(\hat{\vc \pi} \otimes \vc w) N (\hat P \otimes I) =
(\hat{\vc \pi} \otimes \vc w) (\hat P \otimes I) =
\hat{\vc \pi} \otimes \vc w.
\]
We conclude, on the one hand, that $\vc\pi^{(1)}=\hat{\vc \pi} \otimes \vc w$  and, on the other hand, that
$\vc\pi^{(1)} (I-N) = \vc 0$.
\end{proof}

%
%

\section*{Acknowledgements} Sophie Hautphenne is supported by the
Australian Research Council (ARC) grant DP110101663.  Guy Latouche
acknowledges the support of the Minist\`ere de la Communaut\'e
fran\c{c}aise de Belgique through the ARC grant AUWB-08/13--ULB~5.

\bibliographystyle{abbrv} \bibliography{MBTRE}

\end{document}